\documentclass{birkjour}
 \newtheorem{thm}{Theorem}[section]
 \newtheorem{cor}[thm]{Corollary}
 \newtheorem{lem}[thm]{Lemma}
 \newtheorem{prop}[thm]{Proposition}
 \theoremstyle{definition}
 
 \theoremstyle{remark}

 \numberwithin{equation}{section}

\begin{document}

%
%
%
%
%
%
%
%
%

\title[ Biharmonic CMC Hypersurfaces]
 { Biharmonic CMC Hypersurfaces in the Warped Product Manifolds}

\author[N. Mosadegh]{N. Mosadegh}

\address{Department of Mathematics,
 Azarbaijan Shahid Madani University,\\
Tabriz 53751 71379, Iran}

\email{n.mosadegh@azaruniv.ac.ir}

\author{E. Abedi}
\address{Department of Mathematics,
 Azarbaijan Shahid Madani University,\\
Tabriz 53751 71379, Iran}
\email{esabedi@azaruniv.ac.ir}
\subjclass{Primary 53C42; Secondary 53C43, 53B25}

\keywords{Biharmonic hypersurfaces, Warped product manifolds.}

\date{January 1, 2004}
\dedicatory{}

\begin{abstract}
We find the necessary and sufficient condition for the proper biharmonic CMC
hypersurfaces in the special warped product space. Furthermore, we obtain that there exists no
proper biharmonic CMC compact hypersurface there.
\end{abstract}
\footnotetext[1]{The second author is as corresponding author.}
\maketitle
\section{Introduction}

A harmonic map $f: M \rightarrow N $ between two Riemannian manifolds, where $M$ is compact, is considered as the critical points of the energy functional
\begin{eqnarray*}
&&E: C^{\infty}(M,N)\rightarrow R\nonumber\\
&& E(f)=\frac{1}{2}\int_M |df|^2d\upsilon,
 \end{eqnarray*}
where $d\upsilon$ is the volume form of $M$. By taking the similar idea, the notion of $k$-harmonic maps was introduced by Eells and Lemair \cite{J_2}, where the problem was proposed to investigate $k$-harmonic maps as the critical points of the $k$-energy functional see \cite{J}. In case $k=2$, the bienergy of $f$ was defined by
\begin{eqnarray*}
 E_2(f)=\frac{1}{2}\int_M |\tau(f)|^2 d\upsilon,
\end{eqnarray*}
where $\tau(f)= \textsf{trace}\nabla df$ is known the tension field of $f$.  In \cite{G,G_2}, the first variation formula of the bienergy and the definition of $2$-harmonic maps were derived in point of the variational formulas view, in which 
\begin{eqnarray}\label{5.36}
\tau_2(f)=-\Delta \tau(f)- \textsf{trace} R^N(df(.), \tau(f))df(.)=0,
\end{eqnarray}
where the equation $\tau_2(f)=0$ is named the biharmonic equation.

Independently, the biharmonic submanifolds in the Euclidean space with the harmonic mean curvature vector was established by B. Y. Chen. Indeed, the following well known conjecture was posted in \cite{B}: any biharmonic submanifold in Euclidean space is harmonic. Actually, Chen's notion can be recovered according to the property of the biharmonic maps for the Riemannian immersions into the Euclidean spaces. Later on, in \cite{Dim} the author found every biharmonic submanifold with the constant mean curvature is minimal in the Euclidean space. Obviously, any harmonic map is biharmonic, interesting is in the non harmonic biharmonic maps which are named proper biharmonic. In this sense, the biharmonic hypersurfaces, those are the special case of biharmonic submanifolds, have been extensively studied in the last decade. The proper biharmonic hypersurfaces were considered with at most three distinct principal curvatures in $4$-dimensional space form and taken a classification of them in \cite{Bol}. Also, some results concerning the biharmonic hypersurfaces are obtained by the investigating on the number of distinct principal curvatures. The  biharmonic hypersurfaces with three distinct principal curvatures in the Euclidean space are minimal, see \cite{YU1,YU3}. With regarding this idea, in \cite{Abedi2} the authors obtained some non-existence result of the biharmonic Ricci Soliton hypersurfaces in the Euclidean spaces too.  

Additionally, in spaces of the non-constant sectional curvature there exist several classification results concerning the proper biharmonic hypersurfaces, which have been investigated in works for instance see \cite{Oni,JI,R1,R2}. For example, all the proper biharmonic Hopf cylinders in the $3-$dimensional Sasakian space forms were classified. Morevere, all the proper-biharmonic Hopf cylinders over a homogeneous real hypersurfaces in the complex space form spaces were determined. Also, in \cite{Abedi3} the authors deal with the biharmonic pseudo Hopf hypersurfaces in the special Sasakian space form $S^{2m+1}$ and obtained some existence and non existence results about them, where the grad$H$ has significant role in the tangent bundle. Now, in this note we consider the biharmonic hypersurfaces, which are influenced of the conformal change of the metric in the warped products. The concept of a new product manifold was established by Bishop and O'Neill \cite{O}, what is called warped product. Actually, new Riemannian manifolds were illustrated by the warped metric. The authors in \cite{L} showed that, the generalized Robertson-Walker space-time and the standard static space-time are two well known warped product manifolds. So, due to the  application of the warped manifold, they have significant role either in the differential geometry or the theoretical physic.

The present article is included the result of studying about the biharmonic hypersurfaces in the almost contact metric manifold which is written as a warped product manifold $M_1 \times_f M_2$, where $M_1$ and $M_2$ are the Sasakian space form and complex space form, respectively. In section $3$ the necessary and sufficient condition of the biharmonic hypersurface $M^{2(m+n)}$ is determined. Indeed, the biharmonic condition is specialized into two cases, where the unit normal vector field $N \in T^\perp(M^{2(m+n)})$ is either in $\pounds (M_1)$ or $\pounds (M_2)$ which are the horizontal and the vertical lifts, respectively. Also, the proper biharmonic hypersurfaces with the constant mean curvature(CMC) in $M_1\times_f M_2$ is determined. Furthermore, it is shown that if the warped map $f$ satisfies an inequality $|\nabla f|^2 \leq \frac{2n-1}{2n}f^2 \mathsf{Area}(M^{2(m+n)})$, then the proper  biharmonic oriented compact CMC hypersurface does not exist.

 \section{preliminaries}
 In this section, we recall some fundamental definitions and illustrate the geometry of the ambient manifold in order to reach the theorem and the results about the biharmonic hypersurfaces $M^{2(m+n)}$ into the warped product manifold $M_1\times_f M_2$, where $M_1$ and $M_2$ are the Sasakian space form and complex space form, respectively.

Let $M_1$ and $M_2$ be two Riemannian manifolds, which are equipped with the Riemannian metrics $g_1$ and $g_2$, respectively and let $f$ be a positive smooth function on $M_1$. Now, the product manifold $M_1\times_f M_2$ with the Riemannian metric such that
\begin{eqnarray}
g_f= \pi^\ast g_1 +f^2 \sigma^\ast  g_2,
\end{eqnarray}
 is called a warped product manifold $M_1\times_f M_2$, where $\pi: M_1\times M_2 \longrightarrow M_1$ and $\sigma: M_1\times M_2\longrightarrow M_2$ are the natural projection. Also, $f$ is named a warping map of the warped manifold.

 The tangent vectors in $M_1\times \{q\}=\sigma^{-1}(q)$ and $\{p\}\times M_2= \pi^{-1}(p)$ are called horizontal and vertical vectors, respectively; where $q\in M_2$ and $p\in M_1$.
For a vector field $X \in \chi(M_1)$, the horizontal lift of $X$ to $M_1\times_f M_2$ is a vector field $\overline{X}$ such that $\pi_\ast(\overline{X})=X$. The set of all the horizontal and the vertical lifts are denoted by $\pounds(M_1)$ and $\pounds(M_2)$, respectively.

The following Lemma shows, how the Levi-Civita connection $\nabla^f$ of a warped product manifold $M_1\times_f M_2$ is related to the Levi-Civita connections $\nabla^1$ and $\nabla^2$ of $M_1$ and $M_2$, respectively.
\begin{prop}\label{5.21}
\cite{B2} For $X,Y\in \pounds(M_1)$ and $V,W \in \pounds(M_2)$, we have on $M_1\times_f M_2$ that
 \begin{itemize}
   \item $\nabla^f_X Y \in \pounds(M_1)$ is the lift of $\nabla^1_X Y$ on $M_1$,
   \item $\nabla^f_X V = \nabla^f_V X = (X ln f)V$,
   \item nor$(\nabla^f_V W) = h(V,W) = \frac{-<V,W>}{f}\nabla f$,
   \item tan$(\nabla^f_V W)\in \pounds(M_2)$ is the lift of $\nabla^{2}_V W$ on $M_2$, where $\nabla^{2}$ is the Levi-Civita connection of $M_2$.
 \end{itemize}
\end{prop}

An odd dimensional manifold $M_1$ is equipped with the tensor fields $\varphi$, $\xi$ and $\eta$ of types $(1, 1)$, $(0, 1)$ and $(1, 0)$, respectively is called an almost contact manifold, where the following conditions satisfy
\begin{eqnarray*}
\varphi^2(X)=-X+ \eta(X)\xi,\ \ \
\eta(\xi)=1,\ \ \  \eta(\varphi X)=0,
\end{eqnarray*}
for $X \in T(M_1)$. The triple $(\varphi, \xi, \eta)$ is named an almost contact structure too. Now, $M_1$ is endowed a Riemannian metric $g_1$, so $( M_1,\varphi, \xi, \eta, g_1)$ is called an almost contact metric manifold provided that
\begin{eqnarray}
\eta(X)=g_1(\xi, X),\ \ \ 
 g_1(\varphi X, \varphi Y)= g_1(X, Y)- \eta(X)\eta(Y),
\end{eqnarray}
where $X$ and $Y$ in $T(M_1)$. Also, the almost contact metric manifold is called a contact metric manifold, where $g_1(X, \varphi Y)=d\eta(X, Y)$. A contact metric manifold $(M_1, \varphi, \xi, \eta, g_1)$ is normal if $N_\varphi + d\eta \otimes \xi=0$, where
\begin{eqnarray}
N_\varphi (X, Y)=[\varphi X, \varphi Y]-\varphi[\varphi X, Y]-\varphi[X, \varphi Y]+ \varphi^2[X, Y],
\end{eqnarray}
 is the Nijenhuis tensor field of $\varphi$, in this case $M_1$ is called a Sasakian manifold. In other words, the necessary and sufficient condition for a contact metric manifold $(M_1, \varphi, \xi, \eta, g_1)$ to be a Sasakian manifold, written as:
\begin{eqnarray}
 (\nabla^1_X \varphi)Y=g_1(X,Y)\xi- \eta(Y)X, \ \ \
\end{eqnarray}
where $\nabla^1$ is the Levi-Civita connection on $M_1$.

For a Sasakian manifold $(M_1, \varphi, \xi, \eta, g_1)$ the sectional curvature of $2$-plane spanned by $\{X, \varphi X\}$ is called $\varphi$-sectional curvature, where $X$ is orthogonal to $\xi$. A Sasakian manifold, which has the constant $\varphi$-sectional curvature $c$ is called a Sasakian space form and determined by $M_1(c)$. Then its Riemannian tensor field satisfies
\begin{eqnarray}
R^1(X, Y)Z &=& -\frac{c-1}{4}\{\eta(Z)[\eta(Y)X - \eta(X)Y] \nonumber \\
& & + [g_1(Y, Z)\eta(X)- g_1(X, Z)\eta(Y)]\xi\nonumber \\
& & + g_1( \varphi X, Z)\varphi Y + 2g_1(\varphi X, Y)\varphi Z - g_1(\varphi Y,Z)\varphi X\}  \\
& & + \frac{c+3}{4}\{ g_1(Y, Z)X - g_1(X, Z)Y\} \nonumber.
\end{eqnarray}
Also, we take the complex Kahler manifold $M_2$ into account with the known curvature tensor as the following
\begin{eqnarray}
R^2(X, Y)Z &=& k\{g_2(Y, Z)X - g_2(X, Z)Y \nonumber\\
& & + g_2(JY , Z)JX - g_2(JX, Z)JY- 2g_2(JX, Y)JZ\},
\end{eqnarray}
for $X, Y$ and $Z\in T(M_2)$, where $g_2$ is the Riemannian metric on $M_2$.

Now, we make clear the geometry of the ambient warped product manifold $M_1 \times_f M_2$. At first, we show that $M_1\times_f M_2$ has an almost contact metric structure. We give a $(1,1)$-tensor field $\overline{\varphi}$, a vector field $\overline{\xi}$  and a $1$-form $\overline{\eta}$ on $M_1\times_f M_2$, for each $X, Y \in T(M_1\times_f M_2)$, in which
\begin{eqnarray}\label{5.20}
   \overline{\xi}=(\xi,0), \ \ \      \overline{\eta}(X)=g_f(X,\overline{\xi}), \ \ \
\overline{\varphi}X= (\varphi \pi_* X, J\sigma_* X),
\end{eqnarray}
then we have
\begin{eqnarray}
 \overline{\varphi}^2 X= -X+ \overline{\eta}(X)\overline{\xi}   ,\ \ \ \      \overline{\varphi}(\overline{\xi})=0.
\end{eqnarray}
Also, the warped product manifold $M_1\times_f M_2$ admits a contact metric, written as
\begin{eqnarray}\label{5.22}
g_f(\overline{\varphi} X,\overline{\varphi} Y)=g_f (X,Y)-\overline{\eta}(X)\overline{\eta}(Y),
\end{eqnarray}
so $(M_1\times_f M_2, \overline{\varphi}, \overline{\xi}, \overline{\eta}, g_f)$ is an almost contact metric manifold.

Furthermore, we prove that the ambient manifold $M_1\times_f M_2$ is not a Sasakian manifold.
\begin{lem}
For a warped product $M_1\times_f M_2$ we have
\begin{eqnarray}
(\nabla^{f}_X \overline{\varphi})Y &=& \nabla^{f}_X \overline{\varphi}Y- \overline{\varphi}\nabla^{f}_X Y\nonumber\\
 &=& g_1(\pi_\ast X, \pi_\ast Y)\xi- \overline{\eta}(\pi_\ast Y)\pi_\ast X\nonumber\\
 && + (\varphi(\pi_\ast Y)lnf)\sigma_\ast X \nonumber - (\pi_\ast Y lnf)J\sigma_\ast X\\
 && -\frac{g_f(\sigma_\ast X, J\sigma_\ast Y)}{f}\nabla f + \frac{g_f(\sigma_\ast X, \sigma_\ast Y)}{f}\varphi \nabla f,\nonumber
\end{eqnarray}
where $\nabla^f$ stands for the Levi-Civita connection on manifold $M_1\times_f M_2$ for $X, Y \in T(M_1\times_f M_2)$.
\end{lem}
\begin{proof}
For $X, Y$, which are tangent on $M_1\times_f M_2$, by applying the equations $(\ref{5.20})$, $(\ref{5.22})$ and taking into account the Proposition $\ref{5.21}$ we get the result.
\end{proof}
So, however the product manifold has an almost contact metric structure, it is not a Sasakian manifold. Similarly, the equations $(\ref{5.20})$, $(\ref{5.22})$ and the Proposition $\ref{5.21}$ yield
\begin{eqnarray}
\nabla^f_X \overline{\xi}= -\varphi \pi_\ast X + (\xi lnf)\sigma_\ast X.\nonumber
\end{eqnarray}

\section{Biharmonic hypersurface in warped product manifold}

 In this section, we are going to determine a biharmonic hypersurface $M^{2(m+n)}$ in the special warped product $M_1\times_f M_2$, where $M_1$ and $M_2$ are the Sasakian space form and complex space form, respectively. Actually, we specialize the biharmonic concept into the two cases, where the normal vector field $N \in T^\perp(M^{2m+2n})$ is either in $\pounds (M_1)$ or $\pounds (M_2)$. Also, we suppose that $\xi\in T(M_1)$ is tangent to $M^{2(m+n)}$. Then, we get the characterization of the second fundamental form and the scaler curvature of the biharmonic hypersurfaces there.
\begin{thm}\label{5.25}
Let $\psi: (M^{2(m+n)},g)\longrightarrow (M_1\times_f M_2,g_f)$ be an isometric immersion of a hypersurface $M^{2(m+n)}$ in the warped product space $M_1\times_f M_2$. Let $ H \in \pounds(M_1)$, then $M^{2(m+n)}$ is a biharmonic hypersurface if and only if
\begin{eqnarray}
&&\left\{
     \begin{array}{ll}
        \hbox{$\Delta^\perp H = \textsf{trace} B((.), A_H (.))-\frac{c(m+1)+3m -1}{2}H

           +\frac{2n}{f}(\nabla_H\nabla f)^\perp$ }\\
           \hbox{$- 2\textsf{trace} A_{\nabla^\perp_{(.)} H}(.) -(m+n)\textsf{grad}|H|^2-\frac{4n}{f}(\nabla_H\nabla f)^\top=0,$}
         \end{array}
  \right.
\end{eqnarray}
where $A, B$ and $H$ denote the shape operator, the second fundamental form and the mean curvature vector of $M^{2(m+n)}$ into the warped space $M_1\times_f M_2$, respectively.
\end{thm}
\begin{proof}
 Let the mean curvature vector $H= |H|N \in \pounds(M_1)$, where $N$ is a local unit normal vector field of $M^{2(m+n)}$. We put $V=-\overline{\varphi}N$, which is tangent to $M^{2(m+n)}$. Now, we consider a local parallel frame field $\{e_\alpha\}$ at $p\in M^{2(m+n)}$ such that for $1 \leq\alpha\leq 2m$ and $2m+1 \leq \alpha \leq 2(m+n)$ then $e_\alpha \in\pounds(M_1)$ and $e_\alpha\in \pounds(M_2)$, respectively. Also, $\nabla^f$, $\nabla^M$ denote the Levi-Civita connection on $M_1\times_f M_2$ and hypersurface $M^{2(m+n)}$, respectively. Then, by applying the equation $(\ref{5.36})$ we have
\begin{eqnarray}\label{5.0}
  0&=&\tau_2(\overrightarrow{\psi})\nonumber\\
   &=& -\Delta \overrightarrow{H}- \textsf{trace} R^f(d\psi(.), H)d\psi(.)\nonumber\\
   &=&\sum _\alpha \nabla^f_{e_\alpha} \nabla^f_{e_\alpha} H - \textsf{trace} R^f(d\psi(.), H)d\psi(.),
\end{eqnarray}
such that, by applying the Guass and the Wiengarten formulas, we get
\begin{eqnarray}\label{5.4}
\sum _\alpha \nabla^f _{e_\alpha} \nabla^f_{e_\alpha} H &=&\sum _\alpha \nabla^\perp_{e_\alpha}\nabla^\perp_{e_\alpha}H - \sum _\alpha B(e_\alpha, A_He_\alpha)\nonumber\\
                                                        &&- \sum _\alpha A_{\nabla_{e_\alpha}^\perp H} e_\alpha - \sum _\alpha \nabla^M_{_\alpha} A_{H}e_\alpha ,
\end{eqnarray}
where
\begin{eqnarray}\label{5.1}
					\textsf{trace} \nabla^M _{(.)}A_H(.)&=& \sum _{\alpha} \nabla^M _{e_\alpha}(A_H e_\alpha)\nonumber=\sum _{\alpha,\beta} e_\alpha g\big(A_H e_\alpha, e_\beta\big)e_\beta \nonumber\\
     										&=&\sum _{\alpha,\beta} e_\alpha g\big(B(e_\alpha,e_\beta),H \big)e_\beta \nonumber\\
											&=& \sum _{\alpha,\beta} e_\alpha g\big(\nabla^f _{e_\beta} e_\alpha- \nabla^M _{e_\beta} e_\alpha , H\big)e_\beta \nonumber\\
											&=& \sum _{\alpha,\beta} \big \{ g\big(\nabla^f _{e_\alpha}\nabla^f _{e_\beta} e_\alpha, H \big) + g\big( \nabla^f _{e_\beta} e_\alpha, \nabla^f _{e_\alpha} H\big)\big\}e_\beta\nonumber\\
											&=& \sum _{\alpha,\beta} \big\{ g\big(\nabla^f _{e_\alpha}\nabla^f _{e_\beta} e_\alpha, H \big)\nonumber\\
                                            &&+ g\big( \nabla^M _{e_\beta} e_\alpha + B(e_\alpha, e_\beta), -A_H e_\alpha+ \nabla^\perp_{e_\alpha} H\big)\big \}e_\beta\nonumber\\
											&=& \sum _{\alpha,\beta} \big\{g\big(\nabla^f _{e_\alpha}\nabla^f _{e_\beta} e_\alpha, H\big)e_\beta+ g\big(A_{\nabla^\perp H}e_\alpha, e_\beta\big) e_\beta \big\}\nonumber \\
											&=& \sum _{\alpha,\beta} g\big(\nabla^f _{e_\alpha}\nabla^f _{e_\beta} e_\alpha, H \big) e_\beta + \textsf{trace} A_{\nabla^\perp H},
\end{eqnarray}
such that
\begin{eqnarray}\label{5.2}
\sum _{\alpha,\beta} g\big(\nabla^f _{e_\alpha}\nabla^f _{e_\beta} e_\alpha, H \big)e_\beta &=& \sum _\alpha \big \{g\big(R^f(e_\alpha, e_\beta)e_\alpha+ \nabla^f _{e_\beta}\nabla^f _{e_\alpha} e_\alpha, H \big)e_\beta \big \}\nonumber\\
                                                                                    &=& \sum _{\alpha,\beta } g\big(R^f(e_\alpha, e_\beta)e_\alpha, H\big)e_\beta \nonumber\\
                                                                                    &&+ \sum _\alpha g\big(\nabla^f _{e_\beta}(\nabla^M _{e_\alpha} e_\alpha + B(e_\alpha,e_\alpha)),H\big)e_\beta\nonumber\\
                                                                                    &=&\sum _{\alpha,\beta } g\big(R^f(e_\alpha, e_\beta)e_\alpha, H\big)e_\beta\nonumber\\
                                                                                 &&  + 2(m+n)\sum _{\beta} g\big(\nabla^f _{e_\beta} H, H\big)e_\beta\\
                                                                                    &=&\sum _{\alpha,\beta } g\big(R^f(e_\alpha, e_\beta)e_\alpha, H\big)e_\beta
                                                                                 + (m+n)\textsf{grad}|H|^2,\nonumber
\end{eqnarray}
and
\begin{eqnarray}\label{5.33}
\sum _{\alpha,\beta } g\big(R^f(e_\alpha, e_\beta)e_\alpha, H\big)e_\beta &=& \sum _{\alpha,\beta } g\big(R^f(e_\alpha, H)e_\alpha, e_\beta\big)e_\beta\nonumber\\
                                                                 &=& (\textsf{trace}R^f(d\psi(.) , H)d\psi(.))^\top,
\end{eqnarray}
so from the equations $(\ref{5.1})$, $(\ref{5.2})$ and $(\ref{5.33})$ we have
\begin{eqnarray}\label{5.8}
\textsf{trace} \nabla^M _{(.)}A_H (.)&=&(\textsf{trace}R^f((.), H)(.))^\top \nonumber\\
                           && + (m+n)\textsf{grad}|H|^2+ \textsf{trace} A_{\nabla^\perp_{(.)} H}(.),
\end{eqnarray}
now taking into account the equations $(\ref{5.0})$, $(\ref{5.4})$ and $(\ref{5.8})$ we have
\begin{eqnarray}\label{5.32}
0&=&\tau_2(\psi) = \Delta^\perp H - \textsf{trace} B((.), A_H (.)) - (\textsf{trace}R^f(d\psi(.), H)d\psi(.))^\top \nonumber\\
   &&- 2\textsf{trace} A_{\nabla^\perp_{(.)} H}(.) -(m+n)\textsf{grad}|H|^2-\textsf{trace}R^f(d\psi(.), H)d\psi(.).
\end{eqnarray}
In the remainder of the proof, in order to calculate $\textsf{trace}R^f(d\psi(.), H)d\psi(.)$ we take an appropriate orthogonal frame field $\{v_\alpha, \xi, V\}$ on $M^{2(n+m)}$, where $V, \xi$ and $\{v_\alpha\}$ are in $\pounds(M_1)$ for $1 \leq \alpha \leq 2(m-1$) and for $2m+1 \leq \alpha \leq 2(m+n)$, then $v_\alpha \in \pounds(M_2)$. Now, by applying the Proposition $\ref{5.21}$ we get
\begin{eqnarray}\label{5.9}
\textsf{trace}R^f(d\psi(.), H)d\psi(.)&=& \sum_\alpha R^f(d\psi(v_\alpha), H)d\psi(v_\alpha)\nonumber\\
							&&+R^f(d\psi(\xi), H)d\psi(\xi)+ R^f(d\psi(V), H)d\psi(V)\nonumber\\
                            &=&-\frac{c(m+1)+3m -1}{2}H+ \frac{2n}{f}\nabla_H\nabla f,
\end{eqnarray}
so 
\begin{eqnarray}\label{5.31}
\big(\textsf{trace}R^f(d\psi(.), H)d\psi(.)\big)^\top=\frac{2n}{f}(\nabla_H\nabla f)^\top,
\end{eqnarray}
 consequently, take all the equations $(\ref{5.32}), (\ref{5.9})$ and $(\ref{5.31})$ together in which splitting the normal and tangent parts then the result obtain.
\end{proof}
The following result induces immediately
\begin{cor}
 Let $\psi: M^{2(m+n)} \longrightarrow M_1 \times_f M_2$ be an isometric immersion of a hypersurface $M^{2(m+n)}$ in the warped product, $M_1 \times_f M_2$. Let $H\in \pounds(M_1)$ and $|H|=$ constant $\neq 0$, then $M^{2(m+n)}$ is a proper-biharmonic hypersurface if and only if
\begin{eqnarray*}
|B|^2H=\frac{C(m+1)+3m-1}{2}H-\frac{2n}{f}\nabla_H\nabla f.
\end{eqnarray*}
\begin{proof}
By applying the Theorem \ref{5.25} directly, the result obtain.
\end{proof}
\end{cor}
\begin{thm}\label{5.35}
Let $\psi: M^{2(m+n)}\longrightarrow M_1\times_f M_2$ be an isometric immersion of a hypersurface $M^{2(m+n)}$ in the warped product space $M_1\times_f M_2$. Let $H\in \pounds(M_2)$, then $M^{2(m+n)}$ is a biharmonic hypersurface if and only if
\begin{eqnarray}
&&\left\{
                 \begin{array}{ll}
                    \hbox{$\Delta^\bot H =\textsf{trace} B((.), A_H(.))+\frac{1}{f}\Delta f H +(2n-1)(\frac{|\nabla f|^2}{f^2}-1)H$
} \\
                   \hbox{$- 2\textsf{trace} A_{\nabla^\perp_{(.)} H}(.)-(m+n)\textsf{grad} |H|^2=0,$}
                 \end{array}
               \right.
\end{eqnarray}
where $A$, $B$ are the Weingarten operator and the second fundamental form of $M^{2(m+n)}$, respectively.
\end{thm}
\begin{proof}
 At first, we consider an appropriate orthogonal parallel local frame field $\{e_\alpha\}$ at $p\in M^{2(m+n)}$ in such away that for $1\leq \alpha \leq 2m+1$ and  $e_\alpha \in \pounds(M_2)$ where  $2m+2\leq \alpha \leq 2(m+n)$ then $e_\alpha \in \pounds(M_1)$ and, respectively. According to the assumption the mean curvature vector field $H= |H|N\in \pounds(M_2)$. Similarly, the computation goes through which we have done on the Theorem $\ref{5.25}$, in which
\begin{eqnarray}\label{5.10}
\Delta H &=& -\Delta^\perp H + \textsf{trace} B((.), A_H (.))\nonumber\\
         && + \textsf{trace} A_{\nabla^\perp_{(.)} H}(.)+ \textsf{trace} \nabla_{(.)}^{M} A_H(.),
\end{eqnarray}
where
\begin{eqnarray}\label{5.66}
\textsf{trace }\nabla^{M}_{(.)} A_H(.) = \textsf{trace} A_{\nabla^\perp_{(.)} H}(.)+(m+n)\textsf{grad}|H|^2.
\end{eqnarray}
Finally, to investigate the curvature tensor we take an orthogonal frame field $\{v_\alpha\}$ at $p\in M^{2(m+n)}$ where $v_\alpha \in \pounds(M_1)$ and $v_\alpha \in \pounds(M_2)$ for $1\leq \alpha \leq 2m+1$ and $2(m+1)\leq \alpha \leq 2(m+n)$, respectively. By applying the Proposition $\ref{5.21}$ we have
\begin{eqnarray}\label{5.11}
\textsf{trace}R^f(d\psi(.), H)d\psi(.)&=& \sum_\alpha R^f(v_\alpha, H)v_\alpha\nonumber\\
                            &=&\sum_\alpha \{\frac{ H^f(v_\alpha, v_\alpha)}{f}H+ R^{M_2}(v_\alpha, H)v_\alpha\} \nonumber\\
                            &&+ \frac{|\nabla f|^2}{f^2}(2n-1)H\\
                            &=&\frac{1}{f}\Delta f H- (2n-1)H+ \frac{|\nabla f|^2}{f^2}(2n-1)H,\nonumber
\end{eqnarray}
where $H^f (v_\alpha, v_\alpha)= v_\alpha(v_\alpha f)- (\nabla_{v_\alpha} v_\alpha)f$ denotes the Hessian of the warped function $f\in C^\infty (M^{2(m+n)})$ and $X, Y\in T(M^{2(m+n)})$. Then from the equations $(\ref{5.36}), (\ref{5.10}), (\ref{5.66})$ and $(\ref{5.11})$, in which to separate the normal and tangent parts we obtain the result.
\end{proof}
At follow,
\begin{cor}
There exist no proper biharmonic oriented compact CMC hypersurface, into the warped product manifold $M_1\times_f M_2$.
\end{cor}
\begin{proof}
Let the mean curvature $|H|=c\neq 0$ be constant. Then its sufficient to show that the warped map $f$ satisfies
\begin{eqnarray*}
|\nabla f|^2 \leq \frac{2n-1}{2n}f^2 \mathsf{A}(M),
\end{eqnarray*}
where $\mathsf{A}$ is the area function of $M^{2(m+n)}$.

Now, the Theorem $\ref{5.35}$ yields
\begin{eqnarray}\label{5.12}
|B|^2= -\frac{1}{f}\Delta f - (2n-1)\frac{|\nabla f|^2}{f^2}+ (2n-1),\nonumber
\end{eqnarray}
where
\begin{eqnarray*}\label{5.13}
\int_M \frac{1}{f}\Delta f d\upsilon= \int_M \frac{g_f(\nabla f, \nabla f)}{f^2}d\upsilon,\nonumber
\end{eqnarray*}
beacuse $M^{2(m+n)}$ is a oriented compact manifold. So,  
\begin{eqnarray*}
0\leq \int_M |B|^2d\upsilon 
&=& - \int_M \frac{|\nabla f|^2}{f^2} d\upsilon -(2n-1)\int_M \frac{|\nabla f|^2}{f^2}d\upsilon \nonumber\\
&&+(2n-1)\mathsf{A}(M)\nonumber\\
                               &=&-2n \int_M \frac{|\nabla f|^2}{f^2}d\upsilon+(2n-1)\mathsf{A}(M),
\end{eqnarray*}
which yields $|B|^2$=0. In the other words
\begin{eqnarray*}
|\nabla f|^2 \leq \frac{2n-1}{2n}f^2 \mathsf{A}(M).
\end{eqnarray*}
\end{proof}

\begin{prop}
Let $\psi: M^{2(m+n)}\longrightarrow M_1\times_f M_2$ be an isometric immersion of a biharmonic hypersurface $M^{2(m+n)}$ in the warped product manifold $M_1\times_f M_2$.  Then the scaler curvature  of $M^{2(m+n)}$ satisfies
\begin{itemize}
  \item  
  \begin{eqnarray*}
  &s=-\frac{2n}{f}\Delta f + 2n(2n-1)\frac{|\nabla f|^2}{f^2}+ \frac{2n}{f}H^f(N, N)\nonumber\\
   &-\frac{2n}{f}\sum_{\beta=1}^{2m}H^f(e_\beta, e_\beta)- |B|^2 +4(m+n)^2|H|^2 +k,
  \end{eqnarray*}
  where, the local unit normal vector field $N$ is in $\pounds(M_1)$ and $k=\frac{(m-1)(c-1)}{2}+ \frac{(2m-1)(c(m+1)+3m-1)}{2} +4n(n+1)$.
  \item
  \begin{eqnarray*}
  &s= d-2(2n-1)\frac{\Delta f}{f}- \frac{2(n-1)^2}{f^2}|\nabla f|^2
   + (2m+2n)^2|H|^2-|B|^2,
  \end{eqnarray*}
  where, the local unit normal vector field $N$ of $M^{2(m+n)}$ is in $\pounds(M_2)$ and $d=\frac{c-1}{2}(m-n+1) + m(m+1)(c+3)+4n(n+1)+4$.
\end{itemize}
\end{prop}
\begin{proof}
By taking into account the Proposition $\ref{5.21}$ and a straightforward computation, we obtain the results, easily.
\end{proof}

This work has been financially supported by Azarbaijan Shahid Madani Uniyersity
Under the grant number... .

\end{document}